\documentclass[10pt]{amsart}
\usepackage[margin=1in,letterpaper,portrait]{geometry}
\usepackage{latexsym,amscd,amssymb,hyperref,graphicx,xypic,verbatim}

\theoremstyle{plain}
  \newtheorem{theorem}{Theorem}[]
  \newtheorem{proposition}[theorem]{Proposition}

  \newtheorem{conjecture}[theorem]{Conjecture}
\theoremstyle{definition}

 \theoremstyle{remark}
  \newtheorem{remark}[theorem]{Remark}

\numberwithin{equation}{section}

\def\RR{{\mathbb R}}

\def\QQ{\mathbb{Q}}

\def\Sym{{\mathrm{Sym}}}

\def\symm{\mathfrak{S}}

\def\Pconf{\operatorname{Conf}}
\def\Lie{\mathrm{Lie}}

\def\sgn{\mathrm{sgn}}
\def\im{\mathrm{im}}

\def\sphere{\mathbb{S}}

\newcommand\mapsfrom{\mathrel{\reflectbox{\ensuremath{\longmapsto}}}}

\begin{document}

\author{Nicholas Early}
\email{earlnick@gmail.com}
\address{Massachusetts Institute of Technology\\Cambridge MA\\02139}
\author{Victor Reiner}
\email{reiner@umn.edu}
\address{School of Mathematics\\University of Minnesota\\Minneapolis MN 55455}

\title[On Whitehouse's lifts]{On configuration spaces and Whitehouse's lifts of the Eulerian representations}

\begin{abstract}
	S. Whitehouse's lifts of the Eulerian representations of $\symm_n$ to $\symm_{n+1}$ are reinterpreted, topologically and ring-theoretically, building on the first author's work on A. Ocneanu's theory of permutohedral blades.
\end{abstract}

		
\thanks{First, second authors supported by NSF grants DMS-1148634, DMS-1601961.  The paper was written while the first author was an RTG postdoc at the University of Minnesota}

\maketitle

\begingroup
\let\cleardoublepage\relax
\let\clearpage\relax

\endgroup
\section{Introduction}
\label{intro-section}

In their work on Hochschild cohomology \cite{GerstenhaberSchack},
and building on work of Barr \cite[Prop. 2.5]{Barr}, 
Gerstenhaber and Schack introduced certain orthogonal idempotents  
$\{ e^{(j)}_n \}_{i=1,2,\ldots,n }$ inside 
the group algebra $\QQ \symm_n$ of the symmetric group $\symm_n$, 
now called the {\it Eulerian idempotents}.
To define them, write for $i=1,2,\ldots,n-1$ the element
$
s_{i,n-i}:=\sum_{w} \sgn(w) w
$ 
in which the sum ranges over all $w=(w_1,w_2,\ldots,w_n)$ in $\symm_n$
satisfying
$
w_1 < w_2 < \cdots < w_i \text{ and }
w_{i+1} < w_{i+2} < \cdots < w_n.
$
Their sum $s_n:=\sum_{i=1}^n s_{i,n-i}$ turns out to act 
semisimply on $\QQ \symm_n$ via left multiplication, 
with eigenvalues $\{2^j-2\}_{j=1}^n$.  The $j^{th}$ {\it Eulerian idempotent} $e^{(j)}_n$ is the polynomial in 
$s_n$ projecting to the $(2^j-2)$-eigenspace  $E_n^{(j)}:= e^{(j)}_n \QQ \symm_n$.
See Whitehouse \cite[\S1]{WhitehouseInduced} for a nice discussion
of the relation to Hochschild cohomology,
and Reutenauer's book \cite{Reutenauer}, particularly its Section 9.5, 
for history and combinatorial context.

The dimension of $E_n^{(j)}$ counts the 
permutations $w$ in $\symm_n$ with $j$ cycles, that is, it is the (signless)
{\it Stirling number of the first kind}.
The (right-) $\symm_n$-representation on $E_n^{(j)}$ is well-studied, by
Hanlon \cite{Hanlon} and others: 
\begin{itemize}
\item
For $j=n$, the representation
$E_n^{(n)}$ carries the {\it sign} representation $\sgn$.
\item
For $j=1$, the $0$-eigenspace $E_n^{(1)}$, after tensoring with $\sgn$, becomes isomorphic to the well-known 
$\symm_n$-representation $\Lie_n$
on the multilinear part of the {\it free Lie algebra} on $n$ generators.
\item
In general, $\sgn \otimes E_n^{(j)}$ is the multilinear part
in one of the {\it higher Lie characters} in the Poincar\'e-Birkhoff-Witt decomposition for
the tensor algebra in terms of the free Lie algebra.
\end{itemize}

The $E_n^{(j)}$ also have an interpretation in terms
of cohomology of configuration spaces.  Throughout this paper, for
a topological space $X$, all cohomology $H^*X$ will 
mean singular cohomology $H^*(X,\QQ)$ with rational coefficients, and 
$\dim(-)$ will always mean $\dim_{\QQ}(-)$.  
Given a space $X$, its 
$n^{th}$ {\it (ordered) configuration space} $\Pconf^nX$ is the complement of
the thick diagonal within the $n$-fold cartesian product $X^n$:
$$
\Pconf^nX:=\{(p_1,\ldots,p_n) \in X^n: p_i \neq p_j \text{ for }1 \leq i \neq j \leq n\}.
$$ 
The symmetric group action on $X^n$ permuting coordinates preserves $\Pconf^nX$, and hence acts on the cohomology ring $H^*\Pconf^nX$.

It is well-known (see, e.g., Sundaram and Welker \cite{SundaramWelker}) that  for Euclidean spaces $X=\RR^d$ with $d \geq 1$, this cohomology vanishes except
in degrees $k(d-1)$ for $0 \leq k \leq n-1$.  Comparing \cite[Theorem 4.4(iii) at $k=2$]{SundaramWelker}
with known descriptions of $E^{(n)}_j$ (e.g., from Hanlon \cite{Hanlon}; see Sundaram \cite[eqn. (2.16)]{Sundaram-variations})
shows that for $d \geq 3$ odd, one has an $\symm_n$-module isomorphism 
\begin{equation}
\label{Eulerian-topological-interpretation}
E_n^{(j)} \,\, \cong \,\, \sgn \otimes H^{(n-j)(d-1)} \Pconf^n(\RR^d).
\end{equation}
Earlier, F. Cohen (see \cite{Cohen}) 
had given a presentation for the cohomology algebra of 
$H^*\Pconf^n(\RR^d)$, which takes the following form for odd $d \geq 3$.

\vskip.1in
\noindent
{\bf Definition.}
Define a commutative $\QQ$-algebra $\mathcal{U}^n$ via the presentation
\begin{equation}
\label{Cohen-presentation}
\mathcal{U}^n:= \QQ[u_{ij}]_{1 \leq i \neq j \leq n}/I
\end{equation}
where $\QQ[u_{ij}]_{1 \leq i \neq j \leq n}$ is a commutative polynomial ring, and $I$ is the ideal generated by
\begin{itemize}
\item $u_{ij}^2$ and $u_{ij} + u_{ji}$ for $ 1 \leq i < j \leq n$,
\item $u_{ij}u_{jk}+u_{jk}u_{ki}+u_{ki}u_{ij}$ for triples of distinct integers $(i,j,k)$ with $1 \leq i,j,k \leq n$.
\end{itemize} 
The grading on $\QQ[u_{ij}]$ with $\deg(u_{ij})=1$ makes the ideal
$I$ homogeneous, so that $\mathcal{U}^n$ becomes a graded ring.

\vskip.1in
Cohen showed that for odd $d \geq 3$ one has a ring isomorphism
\begin{equation}
\label{Cohen-isomorphism}
\mathcal{U}^n \cong H^*\Pconf^n\RR^d. 
\end{equation}
This isomorphism multiplies degrees by $d-1$, because it
sends the element $u_{ij}$ in degree one of $\mathcal{U}^n$ to the class in 
$H^{d-1} \Pconf^n\RR^d$
representing the pullback of a nonzero class in $H^{d-1} \Pconf^2(\RR^d)=\QQ$ along
the surjection $\Pconf^n(\RR^d) \rightarrow \Pconf^2(\RR^d)$ that maps $(p_1,\ldots,p_n)$ to $(p_i,p_j)$.
Thus $w$ in $\symm_n$ acts via $w(u_{ij})= u_{w(i),w(j)}$.   
The finitely-presented algebra $\mathcal{U}^n$ defined
in \eqref{Cohen-presentation} is the special case 
for the reflection hyperplane arrangement of type $A_{n-1}$ of what
has been called the {\it Artinian Orlik-Terao algebra}; 
see, e.g. \cite[\S2.2]{MoseleyProudfootYoung}.

We wish to give analogous interpretations to certain idempotents and $\symm_n$-representations defined by Whitehouse  in \cite{WhitehouseInduced}.  Letting $c=(1,2,\ldots,n)$ be an $n$-cycle, she showed that the idempotent 
$
\Lambda_n=\sum_{i=0}^{n-1} \sgn(c)^i c^i 
$ 
of $\QQ \symm_n$ commutes with the element $s_{n-1}$, after embedding $s_{n-1}$ via the inclusion of $\symm_{n-1}$ into $\symm_n$ as the permutations fixing $n$.  Hence $\Lambda_n$ commutes with all Eulerian 
idempotents $\{e_{n-1}^{(j)}\}_{j=1}^{n-1}$, so that the products 
$\{ \Lambda_n e^{(j)}_{n-1} \}_{j=1}^{n-1}$ are also idempotent.  
She then proved several properties of the $\symm_n$-representations 
$$
F^{(j)}_n:= \Lambda_n e^{(j)}_{n-1} \QQ \symm_n = e^{(j)}_{n-1} \Lambda_n \QQ \symm_n  \quad \text{ for }j=1,2,\ldots,n-1.
$$
For one, they {\it lift} the Eulerian representations of $\symm_{n-1}$
to $\symm_n$-representations in this sense \cite[Prop. 1.4]{WhitehouseInduced}: 
\begin{equation}
\label{Whitehouse-lifting-property}
F^{(j)}_n \downarrow^{\symm_n}_{\symm_{n-1}} \cong E_{n-1}^{(j)}.
\end{equation}
She also proved that they have the following virtual description  \cite[Thm. 3.4]{WhitehouseInduced}:
\begin{equation}
\label{Whitehouse-virtual-description}
F^{(j)}_n = \sum_{i=1}^j
   \left( E^{(i)}_{n-1} \uparrow^{\symm_n}_{\symm_{n-1}} - E_n^{(i)} \right).
\end{equation}

Similarly to the Eulerian representations $E^{(j)}_n$ for $j=n$ and $j=1$,
the representation $F^{(n-1)}_n$ is the sign representation of $\symm_n$, while
$\sgn \otimes F^{(1)}_n$ is known as the 
{\it Whitehouse representation}\footnote{It is called the {\it tree representation} in work of Robinson and Whitehouse \cite{RobinsonWhitehouse},
and appears elsewhere, e.g., Sundaram \cite{Sundaram}}
 of $\symm_n$. 
Thus \eqref{Whitehouse-lifting-property} implies
$F^{(1)}_n \downarrow^{\symm_n}_{\symm_{n-1}} \cong \sgn \otimes \Lie_{n-1} (\cong E_{n-1}^{(1)})$. 
In work on cyclic operads, Getzler and Kapranov \cite{GetzlerKapranov} showed that
\begin{equation}
\label{Getzler-Kapranov-observation}
(E_n^{(1)} \cong ) \,\, \sgn \otimes \Lie_{n} \cong V^{(n-1,1)} \otimes F^{(1)}_n.
\end{equation}
Here $V^\lambda$ denotes the irreducible  $\symm_n$-representation indexed by
a partition $\lambda$ of the number $n$; in particular, $V^{(n-1,1)}$ 
is the irreducible {\it reflection} representation of $\symm_n$.
Our first goal is to strengthen \eqref{Getzler-Kapranov-observation} to a graded statement that
characterizes Whitehouse's lifts $F^{(j)}_n$, and which is proven in Section~\ref{graded-characterization-section} below.
Consider the Grothendieck ring $R(\symm_n)[[t]]$ of graded $\QQ \symm_n$-modules, 
where $V \cdot t^k$ represents the
class of the representation $V[k]$ which is $V$ considered as living in degree $k$.
\begin{proposition}
\label{graded-characterization}
The equations \eqref{Whitehouse-virtual-description} defining $F^{(j)}_n$ in terms of $E^{(i)}_n$ are equivalent to this relation 
in $R(\symm_n)[[t]]$:
$$
\sum_{j=1}^n E^{(j)}_n t^{n-j} 
 = \left( 1 + t \, V^{(n-1,1)}  \right) \,\,
  \sum_{j=1}^{n-1} F^{(j)}_n t^{n-1-j} 
$$
\end{proposition}
\noindent
In particular, comparing coefficients of $t^{n-1}$ on the two sides of the proposition recovers \eqref{Getzler-Kapranov-observation}; 
see also \S\ref{Mathieu-remark}.

An observation of Moseley, Proudfoot and Young \cite{MoseleyProudfootYoung} 
leads to our next proposition, a configuration space interpretation for 
$F^{(j)}_n$ that parallels the interpretation \eqref{Eulerian-topological-interpretation} for $E^{(j)}_n$.  To state it, recall that the  
{\it special unitary group} $SU_2$ consists of all 
$2 \times 2$ complex matrices $A$ having $\bar{A}^T A = I$ and $\det(A)=1$, 
and is homeomorphic to the $3$-sphere $\sphere^3$.

\vskip.1in
\noindent
{\bf Definition.}
Consider the diagonal action of $SU_2$ on the Cartesian product $(SU_2)^n$ 
and on $\Pconf^nSU_2$, and define $X_n$ to be the quotient space under 
this diagonal action:
$$
X_n:=\Pconf^nSU_2 / SU_2.
$$
\vskip.1in

In fact, if one includes $\RR^3$ into $SU_2$ 
as the subspace $\RR^3 =SU_2 \setminus \{1\} \subset SU_2$, one can check that
$X_n$ is ($\symm_{n-1}$-equivariantly) homeomorphic to 
$\Pconf^{n-1}\RR^3$ via these inverse homeomorphisms:
\begin{equation}
\label{inverse-homeomorphisms}
\begin{array}{rcl}
X_n= \Pconf^nSU_2/SU_2 & \overset{h}{\longrightarrow} & \Pconf^{n-1}\RR^3 \\
  \overline{(p_1,\ldots,p_{n-1},p_n)} &  \overset{h}{\longmapsto}&
     (p_{n}^{-1}p_1,\ldots,p_{n}^{-1}p_{n-1}) \\
 \overline{ (p_1,\ldots,p_{n-1},1) }  &  \overset{h^{-1}}{\mapsfrom} & (p_1,\ldots,p_{n-1}) 
\end{array}
\end{equation}

Our next goal is the following proposition, proven in Section~\ref{MPY-section} below by re-packaging
a proof from \cite{MoseleyProudfootYoung}.

\begin{proposition}
\label{MPY-proposition}
\label{Whitehouse-lifts-topological-interpretation} 
For $j=1,2,\ldots,n-1$, the isomorphism of 
$\symm_{n-1}$-representations from \eqref{Eulerian-topological-interpretation} at $d=3$
$$
E^{(j)}_{n-1} \cong \sgn \otimes  H^{2(n-1-j)} \Pconf^{n-1}\RR^3,
$$
lifts to an isomorphism of $\symm_n$-representations 
$$
F^{(j)}_n \cong \sgn \otimes  H^{2(n-1-j)}X_n.
$$
\end{proposition}

Our main result is Theorem~\ref{subalgebra-theorem} below.
It identifies the cohomology ring $H^*X_n$ as a concrete subalgebra 
$\mathcal{V}^n$ 
of the ring $\mathcal{U}^n \cong H^*\Pconf^n(\RR^3)$ 
from \eqref{Cohen-presentation}.

\vskip.1in
\noindent
{\bf Definition.}
Let $\mathcal{V}^n$ denote the subalgebra of $\mathcal{U}^n$
generated by the elements
$$
v_{ijk}:=u_{ij}+u_{jk}+u_{ki}
$$
for triples of distinct integers $(i,j,k)$ with $1 \leq i,j,k \leq n$.  
\vskip.1in

Using the inclusion $\RR^3 \hookrightarrow \sphere^3 = SU_2$, define a composite $h \circ \pi \circ i$:
\begin{equation}
\label{topological-composite-sequence}
\xymatrix{
\Pconf^n(\RR^3) \ar[r]^-{i}&  \Pconf^nSU_2 \ar[r]^-{\pi} & \Pconf^n(SU_2)/SU_2 =X_n \ar[r]^-{h} &\Pconf^{n-1}\RR^3                                    &
}
\end{equation}

\begin{theorem}
\label{subalgebra-theorem}
For all $n \geq 3$, the $\QQ$-algebra map determined by 
\begin{equation}
\label{map-on-generators}
\begin{array}{rcl}
\QQ[u_{ij}]_{1 \leq i \neq j \leq n-1} &\overset{\varphi}{\longrightarrow} & \QQ[u_{ij}]_{1 \leq i \neq j \leq n} \\
u_{ij} & \longmapsto & v_{ijn}
\end{array}
\end{equation}
has the following properties.
\begin{enumerate}
\item[(i)] $\varphi$ induces an $\symm_{n-1}$-equivariant injective algebra map $\overline{\varphi}: \mathcal{U}^{n-1} \hookrightarrow  \mathcal{U}^n$.
\item[(ii)] The injection $\overline{\varphi}$ maps $\mathcal{U}^{n-1}$ isomorphically onto the subalgebra $\mathcal{V}^n$ inside $\mathcal{U}^n$ generated by $\{v_{ijk}\}$.
\item[(iii)] 
There exists a scalar $c$ in $\QQ^\times$ such that 
$\overline{\varphi}$  is identified with $c \cdot (h \circ \pi \circ i)^*$, where $(h \circ \pi \circ i)^*$ is the $\symm_{n-1}$-equivariant composite of the maps on cohomology in the top row here:
$$
\xymatrix{
H^*\Pconf^n\RR^3  \ar@{=}[d]  
  &  & H^*X_n  \ar[ll]_-{(\pi \circ i)^*}   
       &H^*\Pconf^{n-1}\RR^3 \ar[l]_-{h^*}  \ar@{=}[d] \\
\mathcal{U}^n& & \mathcal{V}^{n}\ar[ll] \ar[u] &\mathcal{U}^{n-1}\ar[l] 
}
$$
\noindent
The middle vertical map is a grade-doubling $\symm_n$-equivariant 
algebra isomorphism, showing that
$$
\left( F^{(n-1-j)}_n \cong \right) H^{2j} X_n \cong (\mathcal{V}^{n})_{j}.
$$
where $(\mathcal{V}^{n})_j$ denotes the $j^{th}$ graded component of the ring $\mathcal{V}^n$.
\end{enumerate}
\end{theorem}

After proving Propositions~\ref{graded-characterization},  \ref{MPY-proposition}, and Theorem~\ref{subalgebra-theorem} in the next three sections, the last section concludes with some remarks on the relation of this paper to results of
Mathieu \cite{Mathieu2} and of d'Antonio and Gaiffi \cite{d'AntonioGaiffi}, and to a conjecture of
Moseley, Proudfoot, and Young \cite{MoseleyProudfootYoung}.

\section{Proof of Proposition~\ref{graded-characterization}}
\label{graded-characterization-section}

Taking coefficients of $t^{n-j}$ on both sides of the equation in the proposition
$$
\sum_{j=1}^n E^{(j)}_n t^{n-j} 
 = \left( 1 + t \, V^{(n-1,1)}  \right) \,\,
  \sum_{j=1}^{n-1} F^{(j)}_n t^{n-1-j}, 
$$
one sees that it is equivalent to assertion for $1 \leq j \leq n$ that 
\begin{equation}
\label{coefficient-of-t^j-relation}
E_n^{(j)} =F_n^{(j-1)} \,\, \oplus \,\, V^{(n-1,1)} \otimes F_n^{(j)}, 
\end{equation}
with conventions $F^{(0)}_n:=0, F^{(n)}_n:=0$.
Use the fact \cite[Exer. 7.81]{Stanley-EC2}
that any virtual $\symm_n$-module $U$ satisfies
$$
V^{(n-1,1)} \otimes U = \left( U \downarrow^{\symm_n}_{\symm_{n-1}} \right) \uparrow^{\symm_n}_{\symm_{n-1}} - U,
$$
to rewrite \eqref{coefficient-of-t^j-relation} in the equivalent form 
$$
\begin{array}{rccccccl}
E_n^{(j)} &=&F_n^{(j-1)}& \oplus& \left( F_n^{(j)} \downarrow^{\symm_n}_{\symm_{n-1}}\right) &\uparrow^{\symm_n}_{\symm_{n-1}}& -&F_n^{(j)} \\
 &=&F_n^{(j-1)}&\oplus&E_{n-1}^{(j)} &\uparrow^{\symm_n}_{\symm_{n-1}}&-&F_n^{(j)}
\end{array}
$$
where the second equality used \eqref{Whitehouse-lifting-property}.  
This last equation can be rewritten as follows
$$
 F_n^{(j)} - F_n^{(j-1)}= E_{n-1}^{(j)} \uparrow^{\symm_n}_{\symm_{n-1}} 
 - E_n^{(j)},
$$
which is equivalent to \eqref{Whitehouse-virtual-description}. 
This completes the proof of the proposition.

\section{Proof of Proposition~\ref{MPY-proposition}}
\label{MPY-section}

The proposition asserts
$
F^{(j)}_n \cong \sgn \otimes  H^{2(n-1-j)}X_n
$
for $j=1,2,\ldots,n-1$.
However, in light of \eqref{Eulerian-topological-interpretation} and
Proposition~\ref{graded-characterization},
one sees that this is equivalent to exhibiting
an isomorphism of graded $\symm_n$-representations:
\begin{equation}
\label{MPY-observation}
H^*\Pconf^n\RR^3
 \cong \left( 1\,\,\, \oplus \,\,\, V^{(n-1,1)}[2] \right) \otimes
            H^*X_n.
\end{equation}
Recall $V^{(n-1,1)}[2]$ means the $\symm_n$-irreducible $V^{(n-1,1)}$
as a graded representation concentrated in degree $2$.

In fact, \eqref{MPY-observation} was already proven by Moseley, Proudfoot,
and Young in \cite[Prop. 2.5]{MoseleyProudfootYoung}; we simply repeat their proof here
for self-containment. Forgetting the $(n+1)^{st}$ coordinate induces $\symm_n$-equivariant maps
$$
\begin{array}{rcl}
(SU_2)^{n+1} & \longrightarrow & (SU_2)^n\\
\Pconf^{n+1}(SU_2) & \longrightarrow &\Pconf^n(SU_2),\\
X_{n+1} &\longrightarrow&X_n.
\end{array}
$$
This last map gives rise to a fibration sequence
$
F \rightarrow X_{n+1} \rightarrow X_n,
$
in which the fiber 
$
F:=SU_2 \setminus \{p_i\}_{i=1}^n
$
is the $3$-sphere $SU_2$ punctured at $n$ points. 
Note that the cohomology $H^*F$ has this simple description:
\begin{itemize}
\item $H^0 F=\QQ$ with trivial $\symm_n$-action, since $F$ is connected, and
\item $H^2 F$ is spanned by
the $n$ cocycles $\{z_i\}$ dual to cycles that go around the 
$n$ punctures $\{p_i\}_{i=1}^n$, permuted by
$\symm_n$,  and satisfying the relation $\sum_{i=1}^n z_i=0$ as in $V^{(n-1,1)}$.
\end{itemize}
Thus as a graded $\symm_n$-representation,
$
H^*(F) = 1\,\,\, \oplus \,\,\, V^{(n-1,1)}[2],
$
the first tensor factor on the right of \eqref{MPY-observation}.

We claim that \eqref{MPY-observation} then follows by applying
the Leray-Serre spectral sequence to 
$
F \rightarrow X_{n+1} \rightarrow X_n.
$
From \eqref{inverse-homeomorphisms}, the base 
$X_n$ is homeomorphic to $\Pconf^{n-1}(\RR^3)$,
and therefore simply connected when $n \geq 2$, since $\Pconf^{n-1}(\RR^3)$ is
the complement of an arrangement of linear subspaces of real codimension three inside $\RR^{3(n-1)}$.
Therefore the fibration spectral sequence
will have no twisted coefficients.  
Furthermore, since both the base $X_n$ and fiber $F$ 
have cohomology only in even degrees,
the spectral sequence collapses at the first page, giving an $\symm_n$-equivariant isomorphism
\begin{equation}
\label{Leray-Serre-consequence}
H^*X_{n+1}
  \cong H^*F \otimes H^*X_n.
\end{equation}
This is exactly the desired $\symm_n$-isomorphism as in  \eqref{MPY-observation},
noting that their left sides agree due to the homeomorphism $h$ in  
\eqref{inverse-homeomorphisms} with $n$ replaced by $n+1$.  This completes the proof of the proposition.

\section{Proof of Theorem~\ref{subalgebra-theorem}}
\label{subalgebra-section}

Each of the parts (i),(ii),(iii) in the theorem will be proven in a separate subsection below.

\subsection{Proof of Theorem~\ref{subalgebra-theorem}(i)}

Part (i) of the theorem asserts that, letting  $v_{ijk}:=u_{ij}+u_{jk}+u_{ki}$, the
$\QQ$-algebra map determined by 
$$
\begin{array}{rcl}
\QQ[u_{ij}]_{1 \leq i \neq j \leq n-1} &\overset{\varphi}{\longrightarrow}&
\QQ[u_{ij}]_{1 \leq i \neq j \leq n}\\
u_{ij}  &\longmapsto& v_{ijn}
\end{array}
$$
induces a $\symm_{n-1}$-equivariant injective algebra map 
$\overline{\varphi}: \mathcal{U}^{n-1} \hookrightarrow \mathcal{U}^n$.

The $\symm_{n-1}$-equivariance follows since $\symm_{n-1}$ permutes the subscripts $1,2,\ldots,n-1$ on the $u_{ij}$.

To see that $\varphi$ induces a well-defined map 
$\overline{\varphi}: \mathcal{U}^{n-1} \longrightarrow \mathcal{U}^n$, start with the 
easy check (left to the reader) that these 
three identities involving $v_{ijk}$ hold in $\mathcal{U}^n$:
\begin{itemize}
\item[(a)] $v_{ijk}^2=0$.
\item[(b)] $v_{ijk}$ is {\it antisymmetric} in $(i,j,k)$, meaning $v_{\sigma(i),\sigma(j),\sigma(k)} = \sgn(\sigma) v_{ijk}$
for all $\sigma$ in $\symm_3$.
\item[(c)] $v_{ij\ell} v_{jk\ell}+v_{jk\ell} v_{ki\ell} +  v_{ki\ell}v_{ij\ell}=0$.
\end{itemize}
The special cases of (a),(b) with $k=n$ and (c) with $\ell=n$
show that $\varphi$ descends to a map  
$\mathcal{U}^{n-1} \longrightarrow \mathcal{U}^n$,
since $\varphi$ sends each defining relation
$u_{ij}^2$ and $u_{ij}+u_{ji}$ and $u_{ij} u_{jk}+u_{jk} u_{ki}+u_{ki} u_{ij}$ 
in $\mathcal{U}^{n-1}$ to zero in $\mathcal{U}^{n}$.

To see 
$\overline{\varphi}: \mathcal{U}^{n-1} \longrightarrow \mathcal{U}^n$ 
is injective, define a $\QQ$-algebra map 
$
\psi: \QQ[u_{ij}]_{1 \leq i \neq j \leq n} \longrightarrow
\QQ[u_{ij}]_{1 \leq i \neq j \leq n-1}
$
by
$$
u_{ij} \longmapsto 
\begin{cases} 
u_{ij} &\text{ if }i,j \leq n-1,\\
0  &\text{ if either }i=n\text{ or }j=n.
\end{cases}
$$ 
This $\psi$ 
satisfies $\psi(v_{ijn})=u_{ij}$ and hence $(\psi \circ \varphi)(u_{ij})=u_{ij}$
for $1 \leq i \neq j \leq n-1$.
On the other hand, it is easy to check that $\psi$ descends to a map
$\overline{\psi}: \mathcal{U}^n \longrightarrow \mathcal{U}^{n-1}$,
as it sends each defining relation 
$u_{ij}^2$ and $u_{ij}+u_{ji}$ and $u_{ij} u_{jk}+u_{jk} u_{ki}+u_{ki} u_{ij}$ 
from $\mathcal{U}^{n}$ to zero in $\mathcal{U}^{n-1}$.
Thus $\overline{\psi} \circ \overline{\varphi}=1_{\mathcal{U}_{n-1}}$, so
$\overline{\varphi}$ is injective\footnote{The authors thank D. Grinberg for suggesting this argument.}.

\subsection{Proof of Theorem~\ref{subalgebra-theorem}(ii)}

Recall part (ii) of the theorem asserts that
the injection $\overline{\varphi}: \mathcal{U}^{n-1} \hookrightarrow \mathcal{U}^n$
is a ring isomorphism from $\mathcal{U}^{n-1}$ onto the subalgebra $\mathcal{V}^n$ inside $\mathcal{U}^n$ 
generated by $\{v_{ijk}\}$.

By construction, since $\varphi$ sends $u_{ij}$ to $v_{ijn}$, the image 
$\im(\overline{\varphi})$ lies in $\mathcal{V}^n$.
To see that $\im(\overline{\varphi})=\mathcal{V}^{n}$, one needs to know that for $1 \leq i,j,k \leq n-1$
one also has $v_{ijk}$ in $\im(\overline{\varphi})$.  
However, we claim this will follow from the $\ell=n$ case
of another identity in $\mathcal{U}^n$ whose easy check is left to the reader:
\begin{itemize}
\item[(d)] $v_{ijk}-v_{ij\ell}+v_{ik\ell}-v_{jk\ell}=0$
\end{itemize}
The $\ell=n$ case of (d) shows that $v_{ijk}=v_{ijn}-v_{ikn}+v_{jkn}=\varphi(u_{ij}-u_{ik}+u_{jk})$, 
lying in $\im(\varphi)$.

\subsection{Proof of Theorem~\ref{subalgebra-theorem}(iii)}

Recall part (iii) of the theorem asserts that the
ring isomorphism $\overline{\varphi}$ from $\mathcal{U}^{n-1}$ onto the subalgebra $\mathcal{V}^n$ inside $\mathcal{U}^n$ 
generated by $\{v_{ijk}\}$ is, up to an overall scaling by $c$ in $\QQ^\times$,
the same as the composite $(h \circ \pi \circ i)^*$ of these maps on cohomology derived  
from the maps in \eqref{topological-composite-sequence}:
$$
\xymatrix{
H^*\Pconf^n(\RR^3)  \ar@{=}[d] &  
  & H^*X_n  \ar[ll]_-{(\pi \circ i)^*}   
       &H^*\Pconf^{n-1}(\RR^3) \ar[l]_-{h^*}  \ar@{=}[d] \\
\mathcal{U}^n& &\mathcal{V}^{n}\ar[ll] \ar[u] &\mathcal{U}^{n-1}\ar[l] 
}
$$
where the middle vertical map is a grade-doubling $\symm_n$-equivariant 
algebra isomorphism $H^* X_n \cong \mathcal{V}^{n}$.  Our strategy will be to first check
this when $n=3$, and then deduce the general case using functoriality.

\vskip.2in
\noindent
{\sf The case $n=3$.}
Here we claim that it suffices to check 
$(\pi \circ i)^*: H^2 X_3 \rightarrow H^2 \Pconf^3(\RR^3)$ 
is not the {\it zero} map.  To see this claim, note that
since $(\pi \circ i)^*$ is $\symm_3$-equivariant, and  
$\dim H^2 X_3=\dim(\mathcal{U}^{2})_1=1$,  
nonzero-ness would imply $(\pi \circ i)^*(H^2 X_3)$ 
is a $1$-dimensional $\symm_3$-stable subspace of 
$
H^2 \Pconf^3(\RR^3) =(\mathcal{U}^{3})_1.
$
One can explicitly decompose $(\mathcal{U}^{3})_1$ into $\symm_3$-irreducibles,
using its definition as the quotient of the 
$\QQ$-span of $\{u_{ij}\}_{1 \leq i \neq j \leq 3}$
by the relations $u_{ij}+u_{ij}=0$.  One can check that 
$$
\begin{array}{rcccl}
(\mathcal{U}^{3})_1 &=& \QQ v_{123} &\oplus& Z \\ 
                   &\cong & V^{(1,1,1)} &\oplus& V^{(2,1)}
\end{array}
$$
in which 
\begin{itemize}
\item $\QQ v_{123}$ 
carries the $1$-dimensional $sgn$ representation $V^{(1,1,1)}$,
by relation (b) above, and
\item
$Z$ carries the $2$-dimensional irreducible representation
$V^{(2,1)}$ as the span of the images of
$z_1, z_2, z_3$ where $z_i:=u_{i1}+u_{i2}+u_{i3}$ (with convention $u_{ii}=0$),
satisfying $z_1+z_2+z_3=0$.
\end{itemize}
Therefore the only $1$-dimensional $\symm_3$-stable subspace of
$(\mathcal{U}^{3})_1$ is
$\QQ v_{123}$, and hence nonzero-ness would force 
$(\pi \circ i)^*(H^2 X_3)=\QQ v_{123}$, implying that  
$v_{123}$ lies in $\im(\pi \circ i)^*$.  
Also $h$ is a homeomorphism, so
$h^*$ must map $(\mathcal{U}^{2})_1 = H^2 \Pconf^2(\RR^3)=\QQ u_{12}$ isomorphically
onto $H^2 X_3$, which $(\pi \circ i)^*$ then maps further onto $\QQ v_{123}$.
Hence there would exist a scalar $c_{123}$ in $\QQ^\times$ so that $\overline{\varphi}$
induced from sending $u_{12}$ to $v_{123}$ would agree with $c_{123} \cdot (h \circ \pi \circ i)^*$
in this case.

To show that $(\pi \circ i)^*: H^2(X_3) \rightarrow H^2 \Pconf^3(\RR^3)$ 
is not the zero map, we use the identification
of $\RR^3$ with the nonidentity elements\footnote{This identification means that we sometimes use both the additive group structure from $\RR^3$ to add/subtract vectors from $\RR^3$ in the same formula where we use the multiplicative group structure from $SU_2$ to multiply/invert them, as in \eqref{part-of-degree-one-composite}. 
} $SU_2 \setminus \{1\} \subset SU_2$ to exhibit a degree one 
self-map of the unit 2-sphere $\sphere^2$ inside $\RR^3$ that
factors through $\pi \circ i: \Pconf^3(\RR^3) \rightarrow X_3$.  

One obtains such a map by precomposing this composite
\begin{equation}
\label{part-of-degree-one-composite}
\begin{array}{rcccccl}
\Pconf^3(\RR^3) & \rightarrow  & X_3 &\rightarrow& \Pconf^2(\RR^3) &\rightarrow &\sphere^2 \\
(p_1,p_2,p_3) & \longmapsto & \overline{(p_1,p_2,p_3)} & \longmapsto &(p_3^{-1} p_1, p_3^{-1}p_2) 
   &  \longmapsto &\frac{p_3^{-1} p_1 - p_3^{-1}p_2 }{|p_3^{-1} p_1 - p_3^{-1}p_2 |}
\end{array}
\end{equation}
with the map
$
\sphere^2  \longrightarrow  \Pconf^3(\RR^3)
$
sending
$v$ to  $(v, 0, p )$
where $p \in \RR^3$ is any particular choice of a vector having length $|p| > 1$.  The result is the map $\psi_p$ with this formula:
$$
\begin{array}{rcl}
\sphere^2 & \overset{\psi_p}{\longrightarrow} & \sphere^2 \\
v & \longmapsto & \frac{p^{-1} v - p^{-1}0 }{|p^{-1} v - p^{-1}0|}.
\end{array}
$$
This map $\psi_p$ is homotopic to the identity map on $\sphere^2$, by sending $p$ to $\infty$ in $\RR^3$, which is $1$ in $SU_2$:
one can calculate $\lim_{p \rightarrow \infty} \psi_{p}(v) = \frac{v-0}{|v-0|}=\frac{v}{|v|}=v$.
Therefore $\psi_p$ has degree one as a self-map of $\sphere^2$, as desired.

\vskip.2in
\noindent
{\sf The general case $n \geq 3$.}
We first show that each $v_{ijn}$ for $1 \leq i<j \leq n-1 $ lies in $\im(\pi \circ i)^*$,
by reducing to the case $n=3$ as follows. Consider
the commutative diagram with vertical maps 
induced by sending $(p_1,\ldots,p_n)$ to $(p_i,p_j,p_n)$
$$
\xymatrix{ 
\Pconf^n(\RR^3) \ar[d] \ar[r]^-{i}& \Pconf^n(SU_2) \ar[d] \ar[r]^-{\pi} & X_n\ar[d] \\
\Pconf^3(\RR^3) \ar[r]^-{i}& \Pconf^{3}(SU_2) \ar[r]^-{\pi} & X_3\\
}
$$
This gives rise to horizontal composite maps $(\pi \circ i)^*$ on cohomology:
$$
\xymatrix{ 
\mathcal{U}^n& H^*X_n \ar[l]_{(\pi \circ i)^*}\\
\mathcal{U}^{3}\ar[u]& H^*X_3 \ar[u] \ar[l]_{(\pi \circ i)^*} \\
}
$$
From the $n=3$ case applied with the triple of indices $(i,j,n)$ replacing $(1,2,3)$, the 
bottom horizontal map $(\pi \circ i)^*: H^*X_3 \rightarrow \mathcal{U}^{3}$ has
$v_{ijn}$ in its image.  From the fact that the classes
$u_{ij}, u_{in}, u_{jn}$ in $\mathcal{U}^n$ are pulled back from the
corresponding classes in $\mathcal{U}^3$, and commutativity of the diagram,
it follows that the top horizontal map $(\pi \circ i)^*: H^*X_n \rightarrow \mathcal{U}^n$ 
also has $v_{ijn}$ in its image.

Once we know that all $v_{ijn}$ lie in $\im(\pi \circ i)^*$, as before, applying the $\ell=n$ case of property (d) above for the $v_{ijk}$ shows
that all $v_{ijk}$ lie in $\im(\pi \circ i)^*$, and hence so does the entire subalgebra $\mathcal{V}^n$.
Since we already know from part (ii) of the theorem that
$$
\dim \mathcal{V}^{n} = \dim \mathcal{U}^{n-1} \left( = \dim H^* \Pconf^{n-1} \RR^3  = \dim H^* X_n \right),
$$
this shows via dimension count that $(\pi \circ i)^*$ maps $H^* X_n$
{\it isomorphically} onto the subalgebra $\mathcal{V}^n$.  
Since $h$ is a homeomorphism, this also means that $(h \circ \pi \circ i)^*$ maps
$H^* \Pconf^{n-1} \RR^3$ isomorphically onto $\mathcal{V}^n$.  Furthermore, the calculation
for $n=3$, replacing indices $(1,2,3)$ with $(i,j,n)$, shows that $(h \circ \pi \circ i)^*$ sends $u_{ij}$ to $c_{ijn} \cdot v_{ijn}$
for some scalars $c_{ijn}$ in $\QQ^\times$.  However, note that $h \circ \pi \circ i$ is $\symm_{n-1}$-equivariant,
and hence so is $(h \circ \pi \circ i)^*$, which forces all of the scalars $c_{ijn}$ to equal
a single scalar $c$.  This completes the proof of Theorem~\ref{subalgebra-theorem}(iii).


\section{Remarks}

\subsection{Relation to work of Mathieu and of d'Antonio and Gaiffi on the ``hidden'' action of $\symm_{n+1}$}
\label{Mathieu-remark}
Note that replacing $n-1$ by $n$ in Theorem~\ref{subalgebra-theorem} 
gives an $\symm_n$-isomorphism $\varphi:\mathcal{U}^n \rightarrow \mathcal{V}^{n+1}$,
which reveals why there is a ``hidden'' $\symm_{n+1}$-action on $\mathcal{U}^n$.
It also gives an alternate explanation of 
Whitehouse's result \eqref{Whitehouse-lifting-property}
that the $\symm_{n+1}$ action on $F^{(j)}_n$ restricts to the action
of $\symm_n$ on $E^{(j)}_{n-1}$. 
This hidden action and our results bear a close relation to previous work of 
Mathieu \cite{Mathieu1, Mathieu2}, and of
d'Antonio and Gaiffi \cite{d'AntonioGaiffi}, which we explain here.  

In \cite[\S6]{Mathieu1}, Mathieu denotes by
 $A_n$ the ring that we have abstractly presented as $\mathcal{U}^n$ 
in \eqref{Cohen-presentation}.  He then 
identifies it as a certain ``limit ring'',
which he denotes $Inv_n(\infty)^*$ in \cite[Thm. 7.6]{Mathieu1} and denotes
as $(S\mathcal{U}^n)^*$ in \cite[\S3]{Mathieu2}.  
In \cite[Lem. 3.3]{Mathieu2}, he extends the natural
action of $\symm_n$ permuting subscripts on the generators $x_{ij}$ for $1 \leq i<j \leq n$
in his $(S\mathcal{U}^n)^*$ to a ``hidden'' action of $\symm_{n+1}$. One can check that his
action is consistent with ours\footnote{Further translating notations, his generator $x_{ij}$ in $SC^*_n$
corresponds to what we would call the generator $v_{i,j,n+1}$ in $\mathcal{V}^{n+1}$.}.  
Bearing these identifications in mind, then
the part of his result \cite[Cor. 4.5]{Mathieu2} pertaining
to $(S\mathcal{U}^n)^*$ is equivalent to our Proposition~\ref{graded-characterization},
which he had already noted in his \cite[Cor. 4.6]{Mathieu2}
strengthens the Getzler and Kapranov result \eqref{Getzler-Kapranov-observation}.
On the other hand, in \cite{d'AntonioGaiffi}, 
d'Antonio and Gaiffi give another, more direct
representation-theoretic construction of the hidden
$\symm_{n+1}$-action on $\mathcal{U}^n$, along with several results
on the $\symm_{n+1}$-irreducible decomposition of this action.

The novelty of our results, compared to these previous works, is
in identifying the direct sum of  
Whitehouse's lifts $\oplus_j F^{(j)}_n$
as both the cohomology $H^*X_n$, and
as the subalgebra $\mathcal{V}^n$ of $\mathcal{U}^n=H^*\Pconf^n\RR^3$.

\subsection{Relation to the Moseley, Proudfoot, and Young conjecture}
They conjecture the following.

\begin{conjecture}\label{thm: MPY Conjecture} \cite[Conj. 2.10]{MoseleyProudfootYoung}
One has a grade-doubling 
isomorphism of $\symm_n$-modules (but not $\QQ$-algebras)
$$
M_n:=\QQ[u_{ij}]_{1 \leq i \neq j \leq n} / K \,\, \cong \,\, H^*X_n 
$$
where $\deg(u_{ij})=1$, and $K$ is the ideal generated by 
\begin{itemize}
\item[(a)] $u_{ij}+u_{ji}$ for $1 \leq i < j \leq n$,
\item[(b)] $u_{ij} u_{jk} + u_{jk} u_{ki} + u_{ki} u_{ij}$ for triples $(i,j,k)$ of distinct integers $1 \leq i,j,k \leq n$,
\item[(c)] $z_1,\ldots,z_n$ where $z_i:=\sum_{j=1}^n u_{ij}$ (with convention $u_{ii}:=0$).
\end{itemize}
Here $\symm_n$ permutes subscripts in the $u_{ij}$.
\end{conjecture}
Proposition~\ref{MPY-proposition} shows that
this conjecture is equivalent to the assertion that the above
algebra $M_n$ carries the $\symm_n$-representation
$\sgn \otimes F^{(n-j)}_n$ in its $j^{th}$ graded component.
Although both our algebra $\mathcal{V}^n$ and this algebra $M_n$ are generated
in degree one by the images of the elements $v_{ijk}:= u_{ij}+u_{jk}+u_{ki}$,
one can check that the ideals of relations satisfied by these generators 
are different in the two rings, and they are {\it not isomorphic as graded algebras}.

\begin{remark}
We repeat here the motivation given in \cite[\S1]{MoseleyProudfootYoung}
for Conjecture~\ref{thm: MPY Conjecture}.  Both the ring
$\mathcal{U}^n$ defined in Section~\ref{intro-section} 
and the ring $M_n$ considered in  Conjecture~\ref{thm: MPY Conjecture}
are quotients of the {\it Orlik-Terao algebra} $OT_n$, defined to
be the quotient of $\QQ[u_{ij}]_{1 \leq i \neq j \leq n}$ by the relations (a),(b) above (without (c));
alternatively $OT_n$ is the quotient of $\QQ[u_{ij}]_{1 \leq i \neq j \leq n}$ by
all of the relations in $\mathcal{U}^n$ except for not setting $u_{ij}^2=0$.
This ring $OT_n$ has a topological interpretation as the torus-equivariant intersection homology of
a certain hypertoric variety \cite[Thm. 3.1]{MoseleyProudfootYoung}.
Conjecture~\ref{thm: MPY Conjecture} would completely describe the $\symm_n$-representation 
on $OT_n$ in the following way.  It is known that, as a graded $\symm_n$-representation,
$OT_n$ is isomorphic to a graded tensor product 
\begin{equation}
\label{graded-tensor-product}
OT_n \cong M_n \otimes R_n := 
\bigoplus_{d=0}^\infty \,\, \bigoplus_{i+j=d} (M_n)_i \otimes (R_n)_j
\end{equation}
of the quotient algebra $M_n$ with the subalgebra 
$R_n=\QQ[z_1,\ldots,z_n]$ generated by the elements in (c).
Since this subalgebra $R_n$ is isomorphic to the symmetric algebra $\Sym(V^{(n-1,1)})$
of the irreducible reflection representation $V^{(n-1,1)}$ for $\symm_n$, 
the description of the $\symm_n$-representation on $M_n$ coming from
Conjecture~\ref{thm: MPY Conjecture}, along with 
\eqref{graded-tensor-product}, would complete the description for $OT_n$.
\end{remark}

\subsection{Presentation and bases for $\mathcal{V}^n$}
Since Theorem~\ref{subalgebra-theorem} gives a ring isomorphism 
$g: \mathcal{U}^{n-1} \rightarrow \mathcal{V}^{n}$ sending $u_{ij}$ to $v_{ijn}$, 
one can use this to give a minimal presentation for $\mathcal{V}^n$, 
coming from the one in \eqref{Cohen-presentation} for $\mathcal{U}^{n-1}$.

Also, there are known monomial bases in the $u_{ij}$ for $\mathcal{U}^{n}$, such as the 
{\it nbc-basis} \cite[Cor. 5.3]{OrlikTerao}, which can then be
mapped forward to exhibit convenient monomial bases in the $v_{i,j,n+1}$ for $\mathcal{V}^{n+1}$.
We recall here the convenient description of the nbc-basis in this case (see Barcelo and Goupil \cite[\S2]{BarceloGoupil}):
the nbc-monomials are the $n!$ monomials obtained as products of one element from each
of the following $n$ sets:
\begin{equation}
\label{type-A-nbc-monomials}
\{1\},\,\,
\{1,u_{12}\}, \,\,
\{1,u_{13},u_{23}\}, \,\,
\{1,u_{14},u_{24},u_{34}\}, \,\, 
\ldots,\,\,
\{1,u_{1n},u_{2n},\ldots,u_{n-1,n}\}.
\end{equation}

On the other hand, 
since the dimension of the $j^{th}$ graded component 
$(\mathcal{U}^{n})_j$ or $(\mathcal{V}^{n+1})_j$ is the number of permutations 
$w$ in $\symm_{n}$ with $n-j$ cycles, one might expect
basis elements indexed by such $w$.  
One has such a basis, motivated by the work in 
\cite{EarlyCanonicalPlateBasis,EarlyBlades}, 
of the following form.
For a permutation $w$ in $\symm_{n}$, write each cycle $C$ of $w$ uniquely as 
$C=(c_1 c_2  \cdots c_\ell)$ with convention $c_1=\min\{ c_1, c_2,\ldots, c_\ell\}$.  Use this
to define elements
$$
\begin{aligned}
{\bf u}(w)&:=\prod_{\text{cycles }C\text{ of }w} 
             u_{ c_1, c_2} \,\, u_{c_2, c_3} \cdots u_{c_{\ell-1} c_\ell},\\
{\bf v}(w)&:=\prod_{\text{cycles }C\text{ of }w} 
             v_{ c_1, c_2, n+1} \,\, v_{c_2, c_3,n+1} \cdots v_{c_{\ell-1} c_\ell,n+1}.
\end{aligned}
$$

\begin{proposition}
One has bases for $(\mathcal{U}^{n})_j$ and $(\mathcal{V}^{n+1})_j$ 
given by $\{ {\bf u}(w) \}$ and $\{ {\bf v}(w) \}$ 
as $w$ runs through all permutations in $\symm_{n}$ with $n-j$ cycles.  
\end{proposition}

\begin{proof}
First, note that for a permutation $w$ with $n-j$ cycles, both
${\bf u}(w)$ and ${\bf v}(w)$ are homogeneous of degree $j$.  Also, 
since $g: \mathcal{U}^{n} \rightarrow \mathcal{V}^{n+1}$
has ${\bf v}(w)=g({\bf u}(w))$, it suffices to show that
$\{ {\bf u}(w) \}_{w \in \symm_{n}}$ is a basis 
for $\mathcal{U}^{n}$.   Since $\dim \mathcal{U}^{n}=n!$, 
by dimension-counting it suffices to check that 
$\{ {\bf u}(w) \}_{w \in \symm_n}$ spans $\mathcal{U}^n$.

We prove this via induction on $n$. The base case $n=1$
is easy: the identity permutation $e$ in $\symm_1$ has
$u(e)=1$ spanning $\mathcal{U}^1$.  In the inductive step,
note that the description \eqref{type-A-nbc-monomials} of the nbc-monomials shows that 
$\mathcal{U}^n$ is spanned by these products where $x$ runs through the elements of $\mathcal{U}^{n-1}$:
$$
\{x,\,\, x \cdot u_{1n}, \,\, x \cdot u_{2n}, \,\, x \cdot u_{3n}, \,\, \ldots, \,\, x\cdot u_{n-1,n} \}.
$$
The inductive hypothesis implies that all such $x$ lie in the span of 
$\{ {\bf u}(w') \}_{w' \in \symm_{n-1}}$.  Thus it suffices to show that every element of the form
${\bf u}(w') \cdot u_{in}$ with $w'$ in $\symm_{n-1}$ and $1 \leq i \leq n-1$ lies in the 
span of $\{ {\bf u}(w) \}_{w \in \symm_{n}}$.  

Let $C$ be the unique cycle of $w'$ that contains the value $i$. 
If $i$ is alone in this cycle $C$, that is $w'(i)=i$, then
$w=w' \cdot (i,n)$ has ${\bf u}(w') \cdot u_{in}= {\bf u}(w)$, and we are done.
If not, say $i$ lies in a non-singleton cycle 
$C=(c_1 c_2  \cdots c_\ell)$, and let $i=c_k$ with 
$1 \leq k \leq \ell$.
Here we show ${\bf u}(w') \cdot u_{in}$ lies in the span of $\{ {\bf u}(w) \}_{w \in \symm_{n}}$
via downward induction on $k$.  In the base case, where $k=\ell$ so that $C=(c_1 c_2  \cdots c_{\ell-1} i)$,
the element $w$ in $\symm_n$ obtained from $w'$ by
replacing $C$ with $C'=(c_1 c_2  \cdots c_{\ell-1} i n)$ satisfies 
${\bf u}(w') \cdot u_{in}= {\bf u}(w)$, so we are done. 
In the inductive step where $k < \ell$, one can use the relation 
$$
u_{c_{k+1},i} \, u_{i,n}  +  u_{i,n} \, u_{n,c_{k+1}} + u_{n,c_{k+1}} \, u_{c_{k+1},i} =0
$$
to rewrite
$${\bf u}(w') \cdot u_{i,n}={\bf u}(w) + {\bf u}(w') \cdot u_{c_{k+1},n},$$
or equivalently
\begin{eqnarray*}
	\mathbf{u}(w) & = & \mathbf{u}(w')\cdot (u_{i,n}-u_{c_{k+1},n}),
\end{eqnarray*}
where $w$ is obtained from $w'$ by replacing the cycle $C$ with
$C'=(c_1 c_2 \cdots c_{k-1} i n c_{k+1} \cdots c_\ell)$, and downward induction applies to ${\bf u}(w') \cdot u_{c_{k+1},n}$,
showing that it lies in the span of $\{ {\bf u}(w) \}_{w \in \symm_{n}}$, completing the proof.
\end{proof}

\section*{Acknowledgements}
The first author is grateful to his doctoral advisor, Adrian Ocneanu, for many interesting discussions related to his work on plates and blades.  He also thanks William Norledge for interesting discussions.  The second author gratefully acknowledges the hospitality of the American Institute of Mathematics workshop on Representation Stability in June 2016.  In particular, he thanks the members of a working group there that included Kevin Casto, Weiyan Chen, Patricia Hersh, Claudiu Raicu, Rita Jimenez-Rolland,  Simon Rubinstein-Salzedo,  Amin Saied, and David Speyer, for their helpful comments, suggestions, and references.  Both authors thank Fran\c{c}ois Bergeron, Sarah Brauner, Darij Grinberg, Nick Proudfoot, and Sheila Sundaram for helpful comments, and they thank an anonymous referee for suggestions that improved the exposition.


\end{document}